\documentclass[11pt]{amsart}

\usepackage{amscd}
\usepackage{amsmath, amssymb}
\usepackage[all]{xy}
\usepackage{amsfonts}
\newcommand{\de}{\partial}

\newcommand{\ddbar}{\sqrt{-1} \partial \overline{\partial}}
\newcommand{\Ric}{\mathrm{Ric}}
\newcommand{\ov}[1]{\overline{#1}}

\newcommand{\tr}[2]{\mathrm{tr}_{#1}{#2}}
\newcommand{\ti}[1]{\tilde{#1}}
\newcommand{\vp}{\varphi}
\newcommand{\vol}{\mathrm{Vol}}

\newcommand{\ve}{\varepsilon}

\renewcommand{\leq}{\leqslant}
\renewcommand{\geq}{\geqslant}
\renewcommand{\le}{\leqslant}

\numberwithin{equation}{section}

\begin{document}
\newtheorem{claim}{Claim}
\newtheorem{theorem}{Theorem}[section]
\newtheorem{lemma}[theorem]{Lemma}
\newtheorem{corollary}[theorem]{Corollary}
\newtheorem{proposition}[theorem]{Proposition}
\newtheorem{question}{question}[section]
\newtheorem{conjecture}[theorem]{Conjecture}
\newtheorem{definition}[theorem]{Definition}

\theoremstyle{definition}
\newtheorem{remark}[theorem]{Remark}
\title[Finite time collapsing]{Finite time collapsing of the K\"ahler-Ricci flow on threefolds}

\author[V. Tosatti]{Valentino Tosatti}
\address{Department of Mathematics, Northwestern University, 2033 Sheridan Road, Evanston, IL 60208}
\email{tosatti@math.northwestern.edu}
\author[Y. Zhang]{Yuguang Zhang}
\address{Yau Mathematical Sciences Center,  Tsinghua University,  Beijing 100084, P.R.China.}
\email{yuguangzhang76@yahoo.com}

\begin{abstract}
We show that if on a compact K\"ahler threefold there is a solution of the K\"ahler-Ricci flow which has finite time collapsing, then the manifold admits a Fano fibration. Furthermore, if there is finite time extinction then the manifold is Fano and the initial class is a positive multiple of the first Chern class.
\end{abstract}
\maketitle

\section{Introduction}
A Fano manifold is a compact complex manifold $X$ with ample anticanonical bundle $-K_X=\Lambda^{\dim X}T^{1,0}X$. Thanks to the Calabi-Yau Theorem \cite{Ya}, a compact complex manifold is Fano if and only if it admits a K\"ahler metric with positive Ricci curvature.

A compact K\"ahler manifold $X$ is said to admit a Fano fibration if there exists a surjective holomorphic map $f:X\to Y$ with connected fibers, where $Y$ is a compact normal K\"ahler space with $0\leq \dim Y<\dim X$ and such that $-K_X$ is $f$-ample (see Definition \ref{deff} below for the definition of $f$-ampleness). In this case the generic fiber of $f$ is a Fano manifold of dimension $\dim X-\dim Y$.

The simplest example of a Fano fibration is when $Y$ is a point, and $X$ is a Fano manifold. Other simple examples are obtained by taking $X=F\times Y$ where $F$ is a Fano manifold and $Y$ is any compact K\"ahler manifold. Fano fibrations are more general than Mori fiber spaces, where one requires in addition that the relative Picard number of $f$ be equal to $1$.

It is easy to see that if $X$ admits a Fano fibration then there exists a K\"ahler metric $\omega_0$ on $X$ such that the (unnormalized) K\"ahler-Ricci flow
\begin{equation}\label{krf}
\left\{
                \begin{aligned}
                  &\frac{\de}{\de t}\omega(t)=-\Ric(\omega(t))\\
                  &\omega(0)=\omega_0
                \end{aligned}
              \right.
\end{equation}
starting at $\omega_0$ has a solution defined only for a finite time $T>0$, and $\vol(X,\omega(t))\to 0$ as $t\to T$ (see right after the proof of Lemma \ref{hart} for details). We say that the K\"ahler-Ricci flow has {\em finite time collapsing}, by which we simply mean that the total volume of the manifold approaches zero in some finite time. It is known that if this happens then necessarily the Kodaira dimension of $X$ is $-\infty$ (see \cite[Proposition 4.2]{CT}).
See also \cite{CT, FIK, Fo, Fo2,IK, LT,So, So2, SSW, SW1, SY, Ti, Ti2, Z3, Z2} and references therein for more results on finite time singularities of the K\"ahler-Ricci flow.

In this article we study the converse question, and we formulate precisely the following conjecture which has been part of the folklore of the subject:
\begin{conjecture}\label{main2}
Let $X^n$  be a compact K\"ahler manifold. Then there exists a K\"ahler metric $\omega_0$ such that the K\"ahler-Ricci flow \eqref{krf} has finite time collapsing if and only if $X$ admits a Fano fibration $f:X\to Y$. In this case, we can write
\begin{equation}\label{fact}
[\omega_0]=-2\pi T c_1(K_X)+f^*[\omega_Y],
\end{equation}
for some K\"ahler metric $\omega_Y$ on $Y$, where $T$ is the maximal existence time of the flow.
\end{conjecture}

Since $Y$ need not be smooth, in this statement a K\"ahler metric on $Y$ is in the sense of analytic spaces (see e.g. \cite{Bi,Fu2, HP, Va}). Also, the factor of $2\pi$ in \eqref{fact} is due to our definition of Ricci form, locally given by $\Ric(\omega)=-\ddbar\log\det(g)$, so that $[\Ric(\omega)]=2\pi c_1(X)$.

Recall now that a $(1,1)$ class $[\alpha]\in H^{1,1}(X,\mathbb{R})$ is called {\em nef} if it lies in the closure of the K\"ahler cone, i.e. if it is a limit of K\"ahler classes. For example, if $\omega(t)$ is a solution of the K\"ahler-Ricci flow on $X$ with $0\leq t<T<\infty$, then the limiting class $[\alpha]=\lim_{t\to T}[\omega(t)]$ is always nef.

We can now state the following conjecture in analytic geometry, which does not mention the K\"ahler-Ricci flow, and which is in fact equivalent to Conjecture \ref{main2} (see Theorem \ref{equiv}):
\begin{conjecture}\label{bpf}
Let $X^n$ be a compact K\"ahler manifold. Then $X$ admits a Fano fibration $f:X\to Y$ if and only if there exists a closed real $(1,1)$ form $\alpha$ with $[\alpha]\in H^{1,1}(X,\mathbb{R})$ nef, with $\int_X\alpha^n=0$ and with
$$[\alpha]+\lambda c_1(X)$$
a K\"ahler class, for some positive real number $\lambda$. In this case, we can write
\begin{equation}\label{fact2}
[\alpha]=f^*[\omega_Y],
\end{equation}
for some K\"ahler metric $\omega_Y$ on $Y$.
\end{conjecture}

In particular, if this conjecture holds, then $[\alpha]$ contains a smooth semipositive representative (cf. \cite[Conjecture 1]{Ti2}, where it is conjectured that this holds even without the assumption that $\int_X\alpha^n=0$).
As explained in section \ref{sectvol}, it is well-known that Conjecture \ref{bpf} is true if $X$ is projective and $[\alpha]\in H^2(X,\mathbb{Q})$, as a simple consequence of the base-point-free theorem \cite{KMM}, and in fact more generally when $X$ is projective and $[\alpha]$ belongs to the real N\'eron-Severi group $NS_{\mathbb{R}}(X)$. So our main interest in this problem is when the manifold $X$ is not projective, or when the class $[\alpha]$ is not in $NS_{\mathbb{R}}(X)$.

As an aside, we remark that an analogous statement as in Conjecture \ref{bpf} should be true when the class $[\alpha]$ satisfies instead $\int_X\alpha^n>0$. In this case there should exist a bimeromorphic morphism $f:X\to Y$ to a compact normal K\"ahler space $Y$ (of dimension $n$) such that \eqref{fact2} holds. This statement follows from the results in \cite{HP} if the extremal face in $\ov{NA}(X)$ of classes which intersect trivially with $[\alpha]$ is in fact a ray (see \cite{HP} for notation), but this is not the case in general.

We also consider the following related conjecture, raised by Tian \cite[Conjecture 4.4]{Ti} (see also \cite{So}).

\begin{conjecture}\label{main}
Let $(X^n,\omega_0)$ be a compact K\"ahler manifold, let $\omega(t)$ be the solution of the K\"ahler-Ricci \eqref{krf}, defined
on the maximal time interval $[0,T)$ with $T<\infty$. Then as $t\to T$ we have
\begin{equation}\label{point}
\mathrm{diam}(X,\omega(t))\to 0,
\end{equation}
if and only if
\begin{equation}\label{fano}
[\omega_0]=\lambda c_1(X),
\end{equation}
for some $\lambda>0$.
\end{conjecture}

Condition \eqref{point} is called {\em finite time extinction}, and Conjecture \ref{main} states that finite time extinction happens if and only if the manifold is Fano and the initial class is a positive multiple of the first Chern class. Of course, as we will see in section \ref{sectext},
condition \eqref{point} implies that $\vol(X,\omega(t))\to 0$ as $t\to T$, i.e. that finite time extinction implies finite time collapsing. In the setting of Conjecture \ref{main2}, finite time extinction corresponds to the case when $Y$ is just a point.

In general if \eqref{fano} holds, then \eqref{point} holds thanks to work of Perelman (see \cite{ST}), who proved the stronger result that under the volume-normalized flow the diameter remains uniformly bounded above.
If $[\omega_0]\in H^2(X,\mathbb{Q})$ (so $X$ is projective), then this conjecture was proved by Song \cite{So}.

Our main result is the following.

\begin{theorem}\label{main3}
If $n\leq 3$ then Conjectures \ref{main2}, \ref{bpf} and \ref{main} hold.
\end{theorem}

In fact, we will show in Section \ref{sectvol} that in general Conjectures \ref{main2} and \ref{bpf} are equivalent (this is essentially elementary), and that Conjecture \ref{main2} implies Conjecture \ref{main},
by modifying the arguments in \cite{So} to show that once we have a Fano fibration, if the base $Y$ is not a point then the diameter of $(X,\omega(t))$ does not go to zero.  We are then reduced to showing Conjecture \ref{bpf}.
The case when $n=2$ is not hard to deal with. For the case when $n=3$, our main technical tool is the very recent completion of the Minimal Model Program for K\"ahler threefolds by H\"oring-Peternell \cite{HP,HP2}, especially their construction of Mori fiber spaces on a bimeromorphic model of a uniruled K\"ahler threefold \cite{HP2}. Using their results, together with some extra arguments due to the fact that the Fano fibrations that we seek to construct are more general than Mori fiber spaces, in Section \ref{sectext} we construct the Fano fibration in the setting of Conjecture \ref{bpf}.\\

In light of these results, it is desirable to study the behavior of the K\"ahler-Ricci flow on the total space of a Fano fibration, with initial metric satisfying \eqref{fact}. This is in general a very hard problem. When $Y$ is a point this amounts to studying the K\"ahler-Ricci flow on Fano manifolds in the anticanonical class.
In general, it is expected (see \cite{SoT, Ti, Ti2}) that as $t\to T$ the evolving metrics $\omega(t)$ converge (in a suitable sense, away from the subvariety $f^{-1}(S)$ where $S\subset Y$ is the critical locus of $f$ together with the singular set of $Y$) to $f^*\omega_Y$ for some K\"ahler metric on $Y\backslash S$. Furthermore $(X,\omega(t))$ is expected to converge in the Gromov-Hausdorff topology to the metric completion of $(Y\backslash S,\omega_Y)$, which should be homeomorphic to $Y$. Lastly, for any given fiber $X_y=f^{-1}(y), y\in Y\backslash S$, the rescaled metrics
$\frac{\omega(t)}{T-t}\big|_{X_y}$ should converge in a suitable sense to a (possibly singular) K\"ahler-Ricci soliton. These results are essentially known when $Y$ is a point (see \cite{CW,PS,ST}), and some progress has been made in the case of projective bundles \cite{Fo, Fo2, SSW, SW1}, but not much more is known in general.\\

{\bf Acknowledgments.} We thank Andreas H\"oring for his crucial help with the proof of Theorem \ref{three}, Aaron Naber for useful discussions, and the referee for helpful comments.
The first-named author is supported in part by NSF grant DMS-1308988 and by a Sloan Research Fellowship, and the second-named author by grant NSFC-11271015.
This work was carried out while the first-named author was visiting the Yau Mathematical Sciences Center of Tsinghua University in Beijing, which he would like to thank for the hospitality.

\section{Finite time collapsing}\label{sectvol}
In this section we prove Conjectures \ref{main2} and \ref{bpf} when $n\leq 3$.

For the moment we work in general dimension $n$, and will restrict to $n\leq 3$ later on.
Recall the following standard definition (see e.g. \cite[p.140]{BS}):

\begin{definition}\label{deff}
Let $f:X\to Y$ be a holomorphic map between compact complex analytic spaces.
We say that a holomorphic line bundle $L$ on $X$ is $f$-ample if there exists $\ell\geq 1$ such that if we consider the coherent sheaf $\mathcal{F}=f_* (\ell L)$ on $Y$, then
the natural map $f^*f_* (\ell L)\to \ell L$ is surjective and defines an embedding
$\Phi:X\hookrightarrow \mathbb{P}(\mathcal{F}):=\mathrm{Proj}_Y (\mathrm{Sym}(\mathcal{F})),$ such that $f=\pi\circ\Phi$ where $\pi:\mathbb{P}(\mathcal{F})\to Y$ is the projection, and so that $\ell L\cong \Phi^*\mathcal{O}_{\mathbb{P}(\mathcal{F})}(1)$. 
\end{definition}

To start, we make the following useful observation, which is well-known in the projective case (see \cite[Proposition II.7.10]{Ha}). We refer the reader for example to \cite{Bi,Fu2, HP, Va} for the definition and basic properties of K\"ahler metrics on compact complex analytic spaces. Unless otherwise stated, the analytic spaces that we consider are reduced, irreducible, but not necessarily normal.

\begin{lemma}\label{hart}
Let $f:X\to Y$ be a surjective holomorphic map with connected fibers, where $X$ is a compact K\"ahler manifold and $Y$ is a compact normal K\"ahler space with $0\leq \dim Y<\dim X$. Then $-K_X$ is $f$-ample if and only if
there exists a K\"ahler metric $\omega_Y$ on $Y$ such that $f^*[\omega_Y]-c_1(K_X)$ is a K\"ahler class on $X$.
\end{lemma}
\begin{proof}
Assume that $-K_X$ is $f$-ample. This clearly implies that the map $f$ is projective, and this in turns implies that $f$ is a K\"ahler morphism (in the sense of \cite{Fu2}). More precisely we can find a metric $h$ on $-\ell K_X$ whose curvature form $\omega$ is positive definite on all the fibers, see e.g. \cite[Lemma 4.19]{Bi}, \cite[Lemma 4.4]{Fu2} or \cite[Proposition II.1.3.1]{Va}, and then it follows that given any K\"ahler metric $\omega_Y$ on $Y$ there exists $A>0$ large enough so that
$Af^*\omega_Y+\omega$ is a K\"ahler metric on $X$ (see again \cite[Lemma 4.4]{Fu2} or \cite[Proposition II.1.3.1]{Va}), which is in the class
$f^*[A\omega_Y]-\ell c_1(K_X)$.

Conversely, if we have that $f^*[\omega_Y]-c_1(K_X)$ is a K\"ahler class on $X$, for some K\"ahler class $[\omega_Y]$ on $Y$, then $-K_X$ is $f$-ample.
Indeed, for every fiber $F$ of $f$ (which may be singular) we have that $-K_X|_F$ is a holomorphic line bundle with a smooth metric with strictly positive curvature form (in the sense of analytic spaces). Grauert's version of the Kodaira embedding theorem for analytic spaces \cite{Gr} (see also \cite[Theorem 1.1]{Bi2}) implies that $-K_X|_F$ is ample, and this is equivalent to $-K_X$ being $f$-ample, see e.g. \cite[Proposition 1.4]{Bi2}.
\end{proof}

We can now show the easy direction of Conjecture \ref{main2}, namely that on every Fano fibration there always exists a solution of the K\"ahler-Ricci flow \eqref{krf} which collapses in finite time. Given $f:X\to Y$ a Fano fibration, by Lemma \ref{hart} there exists a K\"ahler metric $\omega_Y$ on $Y$ such that $[\omega_0]=f^*[\omega_Y]- c_1(K_X)$ is a K\"ahler class on $X$. Then the K\"ahler-Ricci flow starting at any K\"ahler metric $\omega_0$ in this class has a finite time singularity at time $\frac{1}{2\pi}$ (thanks to the cohomological characterization of the maximal existence time of \eqref{krf} given in \cite{Ts,Ts2,TZ}) and the total volume of $X$ goes to zero as time approaches $\frac{1}{2\pi}$, because the limiting class
$[\omega_0]+c_1(K_X)=f^*[\omega_Y]$ satisfies $\int_X(f^*\omega_Y)^n=0$.

We can also show the easy direction of Conjecture \ref{bpf}. Let $f:X\to Y$ be a Fano fibration, and fix $\omega_Y$ a K\"ahler metric on $Y$ (in the sense of analytic spaces) as in Lemma \ref{hart}, so that $f^*[\omega_Y]- c_1(K_X)$ is a K\"ahler class on $X$. Then $\alpha=f^*\omega_Y$ is a smooth nonnegative real $(1,1)$ form on $X$, and so its cohomology class is nef (it is the limit of the K\"ahler classes $[\alpha+\ve\omega_X]$ as $\ve\downarrow 0$, where $\omega_X$ is any K\"ahler metric on $X$), and satisfies $\int_X\alpha^n=0$ and $[\alpha]+c_1(X)$ is a K\"ahler class, as required.

Next we show:
\begin{theorem}\label{equiv}
Conjectures \ref{main2} and \ref{bpf} are equivalent.
\end{theorem}
\begin{proof}
Since we have just shown that the easy directions of both conjectures always hold, it is enough to show that the other directions are equivalent. Assume first that Conjecture \ref{main2} holds, and let $[\alpha]\in H^{1,1}(X,\mathbb{R})$ be a nef class with $\int_X\alpha^n=0$ and with $[\alpha]+\lambda c_1(X)$ K\"ahler for some $\lambda>0$. Fix a K\"ahler metric $\omega_0$ in this class, and consider its evolution by the K\"ahler-Ricci flow \eqref{krf}. The class of the evolved metric $\omega(t)$ is
$$[\omega(t)]=[\omega_0]-2\pi tc_1(X)=[\alpha]+(\lambda-2\pi t)c_1(X)=\left(1-\frac{2\pi t}{\lambda}\right)[\omega_0]+\frac{2\pi t}{\lambda}[\alpha].$$
For $0\leq t<\frac{\lambda}{2\pi}$ this is a sum of a K\"ahler class and a nef class, and so it is K\"ahler, while for $t=\frac{\lambda}{2\pi}$ this equals $[\alpha]$ which is nef but not
K\"ahler since $\int_X\alpha^n=0$. Then the cohomological characterization of the maximal existence time $T$ of \eqref{krf} given in \cite{Ts,Ts2,TZ}
shows that $T=\frac{\lambda}{2\pi}<\infty$ and the limiting class is $[\alpha]$. Therefore the K\"ahler-Ricci flow $\omega(t)$ has finite time collapsing, and by Conjecture \ref{main2}, $X$ admits a Fano fibration. Also by \eqref{fact} we can write $[\omega_0]=-\lambda c_1(K_X)+f^*[\omega_Y],$ for some K\"ahler metric $\omega_Y$ on $Y$, and so $[\alpha]=f^*[\omega_Y]$, i.e. \eqref{fact2} holds.

Assume conversely that Conjecture \ref{bpf} holds, and let $\omega_0$ be a K\"ahler metric on $X$ such that the K\"ahler-Ricci flow \eqref{krf} has finite time collapsing at time $T<\infty$.
The limiting class of the flow is
$$[\alpha]=[\omega_0]-2\pi Tc_1(X),$$
which is nef, satisfies $\int_X\alpha^n=0$, and
$[\alpha]+2\pi Tc_1(X)$ is K\"ahler. Therefore Conjecture \ref{bpf} gives us a Fano fibration $f:X\to Y$, and $[\alpha]=f^*[\omega_Y]$ for some K\"ahler metric $\omega_Y$ on $Y$, which shows \eqref{fact}.
\end{proof}

So we are left to we consider the converse implication in Conjecture \ref{bpf}, and so we assume we have a nef class $[\alpha]$ with $\int_X\alpha^n=0$ and with $[\alpha]+\lambda c_1(X)$ K\"ahler for some $\lambda>0$. Recall that a $(1,1)$ class is called pseudoeffective if it contains a closed positive current, and that a holomorphic line bundle is called pseudoeffective if its first Chern class is (see e.g. \cite{De}).
We have the following simple remark (cf. \cite[Proposition 4.2]{CT}).
\begin{lemma}\label{psef}
If $X^n$ is a compact K\"ahler manifold which has a nef class $[\alpha]$ with $\int_X\alpha^n=0$ and with $[\alpha]+\lambda c_1(X)$ K\"ahler for some $\lambda>0$.
Then $K_X$ is not pseudoeffective, and therefore $\kappa(X)=-\infty$.
\end{lemma}
\begin{proof}
If $K_X$ is pseudoeffective, then so is the class $-\lambda c_1(X)$. The class $[\alpha]$ is therefore the sum of a K\"ahler class and a pseudoeffective class, and so it is big (in the sense that it contains a K\"ahler current \cite{De}). But a nef and big class always has $\int_X\alpha^n>0$, by \cite[Theorems 4.1 and 4.7]{Bo}, and this contradicts our assumption.
Also, in general $K_X$ not pseudoeffective implies $\kappa(X)=-\infty$, since if $\kappa(X)\geq 0$ then some power $K_X^{\otimes\ell}$ ($\ell\geq 1$) is effective, and so $K_X$ is pseudoeffective.
\end{proof}

A well-known conjecture says that if $K_X$ is not pseudoeffective, then $X$ is uniruled (i.e. covered by rational curves). This is proved in \cite{BDPP} in the projective case, and is also known in the K\"ahler case if $n\le 3$ by \cite{Br}.

If we assume now that $X$ is projective and $[\alpha]\in H^2(X,\mathbb{Q})$, so there is an integer $m\geq 1$ such that $[m\alpha]=c_1(L)$ for some holomorphic line bundle $L$, then it is well-known that Conjecture \ref{bpf} holds. Indeed, by assumption we have that $[\alpha]+\lambda c_1(X)$ is a K\"ahler class, and since this is an open condition we may assume that $\lambda=\frac{p}{q}>0$ is positive and rational. Up to increasing $m$ we may assume that $m\lambda\geq 1$ is an integer, and we fix $\omega_0$ a K\"ahler metric in the class $[m\alpha]+m\lambda c_1(X)$. Then we have
$$c_1(L-K_X)=[\omega_0]+(m\lambda-1)c_1(K_X)=\frac{1}{m\lambda}[\omega_0]+\left(1-\frac{1}{m\lambda}\right)c_1(L),$$
is a K\"ahler class (since it is sum of a K\"ahler and a nef class), i.e. $L-K_X$ is ample. The base-point-free theorem (see \cite{KMM})
then shows that $kL$ is base-point-free for some $k\geq 1$, and so it induces a holomorphic map $f:X\to Y$ onto a normal projective variety $Y$, so that $f$ has connected fibers and $kL$ is linearly equivalent to the pullback of an ample divisor under $f$. This implies that $[\alpha]=f^*[\omega_Y]$ for some K\"ahler metric $\omega_Y$ on $Y$, and since we have that $\int_X\alpha^n=0$, this implies that $0\leq \dim Y<\dim X$.
By construction, a multiple of $-K_X$ is linearly equivalent to the sum of an ample line bundle and the pullback of a line bundle from $Y$, and so $f$ is a Fano fibration and \eqref{fact2} holds.

An extension of this argument deals with the more general case when $[\alpha]$ belongs to the real N\'eron-Severi group
$$NS_\mathbb{R}(X)=(H^{1,1}(X,\mathbb{R})\cap H^2(X,\mathbb{Q}))\otimes\mathbb{R}.$$
This is identified with the space of $\mathbb{R}$-divisors modulo numerical equivalence, and is a subspace of $H^{1,1}(X,\mathbb{R})$, in general of smaller dimension.

\begin{proposition}\label{ns}
Conjecture \ref{bpf} holds if $X$ is projective and $[\alpha]\in NS_\mathbb{R}(X)$. In particular, Conjecture \ref{bpf} holds if
$$H^{2,0}(X)=0.$$
\end{proposition}
\begin{proof}
By assumption we have $\alpha=c_1(D)$ where $D$ is a nef $\mathbb{R}$-divisor, and $c_1(D-\lambda K_X)$ is a K\"ahler class for some positive real number $\lambda$, i.e. $A:=D-\lambda K_X$ is an ample $\mathbb{R}$-divisor. Then the base-point free theorem for $\mathbb{R}$-divisors \cite[Theorem 7.1]{HM} shows that $D$ is semiample, in the sense that there exists a surjective morphism $f:X\to Y$ to a normal projective variety $Y$, with connected fibers, and such that $D$ is $\mathbb{R}$-linearly equivalent to $f^*H$ where $H$ is an ample $\mathbb{R}$-divisor on $Y$. Therefore
$[\alpha]$ is the pullback of a K\"ahler class on $Y$, and as before this implies that $\dim Y<\dim X$. We also have that $-K_X$ is $\mathbb{R}$-linearly equivalent to $\frac{1}{\lambda}(A-f^*H)$, and so $-K_X$ is $f$-ample.

For the last item, it is enough to remark that $H^{2,0}(X)=0$ implies that $X$ is projective (an old result of Kodaira), and that $H^{1,1}(X,\mathbb{R})=NS_\mathbb{R}(X)$ (thanks to the Hodge decomposition on $H^2(X,\mathbb{C})$).
\end{proof}

\begin{corollary}\label{two}
Conjecture \ref{bpf} holds when $n=2$.
\end{corollary}
\begin{proof}
Indeed we have $\kappa(X)=-\infty$ by Lemma \ref{psef}, and when $n=2$ this implies that $H^0(X,K_X)=H^{2,0}(X)=0$, and we conclude by Proposition \ref{ns}.
\end{proof}

Finally we deal with the case when $n=3$.

\begin{theorem}\label{three}
Conjecture \ref{bpf} holds when $n=3$.
\end{theorem}
\begin{proof}
We have a nef class $[\alpha]$ with $\int_X\alpha^3=0$ and with $[\alpha]+\lambda c_1(X)=[\omega_0]$ for some K\"ahler metric $\omega_0$ and some $\lambda>0$.
Recall that $K_X$ is not pseudoeffective, by Lemma \ref{psef}. Thanks to the main theorem of \cite{Br} it follows that $X$ is uniruled, i.e. covered by rational curves.
Let $X\dashrightarrow B$ be the MRC fibration of $X$ (constructed in \cite{Ca}, see also \cite{Ca2, KMM2}), which satisfies $\dim B<3$ because $X$ is uniruled.

If $\dim B=0$ then $X$ is rationally connected (and since it is K\"ahler, it must be projective thanks to \cite[Corollaire, p.212]{Ca}), and so $H^{2,0}(X)=0$ and Conjecture \ref{bpf} holds thanks to Proposition \ref{ns}.
If $\dim B=1$ then $X\to B$ is holomorphic everywhere, $B$ is a compact Riemann surface of genus at least $1$, and the general fiber $F$ is rationally connected (and hence again projective). Therefore we have $H^{1,0}(F)=H^{2,0}(F)=0$ and this easily implies that
$H^{2,0}(X)=0$ and so Conjecture \ref{bpf} holds in this case as well.

Therefore we can assume that $\dim B=2$, and the general fiber $F$ of the MRC fibration is isomorphic to $\mathbb{CP}^1$. Following \cite{HP2}, we call a K\"ahler class $[\omega]$ on $X$ normalized if $\int_F \omega=2$. We will apply the results of H\"oring-Peternell in \cite{HP,HP2}, and we are grateful to Andreas H\"oring for his help with the following arguments.

We consider the K\"ahler class $[\omega]=\frac{1}{\lambda}[\omega_0]$, so that $c_1(K_X)+[\omega]=\frac{1}{\lambda}[\alpha]$ is a nef class. Let
$$\mu=\frac{2}{\int_F\omega}>0,$$ so that
$\mu[\omega]$ is normalized. By adjunction we have $\int_F c_1(K_X)=-2$, and since $c_1(K_X)+[\omega]$ is nef we also have
$$\int_F (c_1(K_X)+\omega)\geq 0,$$
and so $\mu\leq 1$.

Assume that $[\omega]$ is not normalized, i.e. that $\mu<1$. Thanks to \cite[Lemma 3.3]{HP2} the class $c_1(K_X)+\mu[\omega]$ is pseudoeffective, and so
$c_1(K_X)+[\omega]=(c_1(K_X)+\mu[\omega])+(1-\mu)[\omega]$ is big. Since this equals $[\alpha]$, we get a contradiction to the fact that $\int_X\alpha^3=0$ (using
again \cite[Theorems 4.1 and 4.7]{Bo} as in Lemma \ref{psef}).

Therefore $[\omega]$ is normalized, and we can then apply \cite[Theorem 1.4]{HP2} and obtain a holomorphic map $f:X\to Y$ onto a normal compact complex surface $Y$,
such that $f$ has connected fibers and a curve $C\subset X$ satisfies $f(C)$ is a point if and only if $\int_C(c_1(K_X)+\omega)=0$. Therefore $-K_X$ is $f$-nef and its restriction to a generic fiber of $f$ is big. Relative Kawamata-Viehweg vanishing for complex spaces \cite[Theorem 2.1]{An} (cf. \cite{Na}) implies that $R^if_*\mathcal{O}_X\cong\mathcal{O}_Y$ for all $i>0$, and then \cite[Theorem 1]{Ko} gives that $Y$ has at worst rational singularities (the proof there uses Grothendieck duality and Grauert-Riemenschneider vanishing, but these have been extended to the analytic setting in \cite{RR, RRV} and \cite{Ta} respectively). This in turn implies that $Y$ is a K\"ahler space, for example by \cite{Fu}.

From the construction of $f$ in the proof of \cite[Theorem 1.4]{HP2}, we obtain a commutative diagram
\begin{equation}
\begin{CD}
 \Gamma @>{q}>>Z\\
@V{p}VV @VV{\nu}V \\
 X @>{f}>> Y
\end{CD}
\end{equation}
where $\Gamma$ is a compact K\"ahler manifold (in \cite{HP2} $\Gamma$ is just a compact analytic space in class $\mathcal{C}$, but we may replace it with a resolution of singularities), $p$ is a modification, $Z$ is a smooth K\"ahler surface, and $\nu$ is the contraction of an effective divisor $E\subset Z$, which is the null locus of a nef and big $(1,1)$ class $[\beta]$ on $Z$, such that $q^*[\beta]=p^*[\alpha]$.

Since $Y$ has rational singularities, it follows from \cite[Lemma 3.3]{HP} that $[\beta]=\nu^*[\gamma]$ for a $(1,1)$ class $[\gamma]$ on $Y$ (see e.g. \cite{HP} for the definition of $(1,1)$ classes on normal analytic spaces). In fact, $[\gamma]$ is a K\"ahler class (in the sense of analytic spaces). To see this, let $K$ be a K\"ahler current on $Z$ in the class $[\beta]$ which is singular only along the null locus of $[\beta]$, which exists thanks to \cite[Theorem 1.1]{CT} (the proof there simplifies vastly in the case at hand, since $\dim Z=2$). Then the pushforward current $\nu_*K$ has local potentials everywhere on $Y$ and belongs to the class $[\gamma]$ thanks to \cite[Lemma 3.4]{HP}, and it is a smooth K\"ahler metric away from $\nu(\mathrm{Exc}(\nu))$, which is a finite set. It is then easy to produce a K\"ahler metric in the class $[\gamma]$, by modifying the local potentials of $\nu_*K$ near each point in $\nu(\mathrm{Exc}(\nu))$ by taking the regularized maximum of the local potential and $|z|^2-C$ in a local chart ($C$ is a sufficiently large constant), see e.g. \cite[Lemma 3.1]{CT} or \cite[Remark 3.5]{HP}.

Since $p^*(f^*[\gamma]-[\alpha])=0$, and $p^*$ is injective, we conclude that $f^*[\gamma]=[\alpha]$, i.e. \eqref{fact2} holds. In other words, we have obtained that
$$-\lambda c_1(K_X)=[\omega_0]-f^*[\gamma],$$
which implies that $-K_X$ is $f$-ample.
\end{proof}

\section{Finite time extinction}\label{sectext}
In this section we prove Conjecture \ref{main} when $n\leq 3$. As we remarked in the introduction, it suffices to show that \eqref{point} implies \eqref{fano}.

First, we make the following general observation.

\begin{lemma}\label{triv}
$(X^n,\omega_0)$ be a compact K\"ahler manifold, let $\omega(t)$ be the solution of the K\"ahler-Ricci \eqref{krf}, defined
on the maximal time interval $[0,T)$ with $T<\infty$, and such that
\begin{equation}\label{point2}
\mathrm{diam}(X,\omega(t))\to 0,
\end{equation}
as $t\to T$. Then we have that
\begin{equation}\label{vol2}
\vol(X,\omega(t))\to 0,
\end{equation}
as well, so that the flow exhibits finite time collapsing.
\end{lemma}
\begin{proof}
As usual let $[\alpha]=[\omega_0]-2\pi Tc_1(X)$ be the limiting class along the flow.
If we had $$\int_X\alpha^n>0,$$ then \cite[Theorem 1.5]{CT} shows that on the Zariski open set $X\backslash \mathrm{Null}(\alpha)$ we have smooth convergence of $\omega(t)$ to a
limiting K\"ahler metric $\omega_T$ on this set. In particular, the diameter of $(X,\omega(t))$ cannot go to zero.
\end{proof}

From now on we assume that \eqref{point} holds.
In particular, thanks to \cite[Proposition 4.2]{CT} (or Lemma \ref{psef}), we have that $K_X$ is not pseudoeffective.
Of course the content of Conjecture \ref{main} is to show that the limiting class $\alpha$ is zero.

The following is the main result of this section:
\begin{theorem}\label{tree}
If Conjecture \ref{main2} holds for $X$, then Conjecture \ref{main} holds as well.
\end{theorem}

In particular, combining this with Theorem \ref{equiv}, Corollary \ref{two} and Theorem \ref{three}, we obtain the proof of Theorem \ref{main3}. We also see that Conjecture \ref{main} holds under the same hypotheses of Proposition \ref{ns} (a fact which was already mentioned in \cite[Remark 1.1]{So}). The proof of this theorem is a modification of the arguments in \cite{So}.

\begin{proof}
Thanks to Lemma \ref{triv} we are in the setup of Conjecture \ref{main2}, and so we have a Fano fibration $f:X\to Y$ such that
$[\omega_0]=-\lambda c_1(K_X)+f^*[\omega_Y],$ for some $\lambda>0$ and some K\"ahler metric $\omega_Y$ on $Y$ (in the sense of analytic spaces).
The maximal existence time for the flow is thus $T=\frac{\lambda}{2\pi}$, and the limiting class is $[\alpha]=f^*[\omega_Y]$.
Then $f^*\omega_Y$ is a smooth semipositive $(1,1)$ form on $X$ (this follows easily from the definition of K\"ahler metrics and holomorphic maps between analytic spaces), in the limiting class $[\alpha]$. We then write
$$\hat{\omega}_t=\frac{1}{T}((T-t)\omega_0+tf^*\omega_Y),$$
which are K\"ahler metrics for all $0\leq t<T$, and we have
$$\omega(t)=\frac{1}{T}((T-t)\omega_0+tf^*\omega_Y)+\ddbar\vp(t),$$
where $\vp(t)$ solves the parabolic complex Monge-Amp\`ere equation
\begin{equation}\label{ma}
\left\{
                \begin{aligned}
                  &\frac{\de}{\de t}\vp(t)=\log\frac{(\hat{\omega}_t+\ddbar\vp(t))^n}{\Omega}\\
                  &\vp(0)=0
                \end{aligned}
              \right.
\end{equation}
and $\Omega$ is a smooth positive volume form with $\ddbar\log\Omega=\frac{1}{T}(f^*\omega_Y-\omega_0)$.
A simple maximum principle argument gives $|\vp(t)|\leq C$ on $X\times[0,T)$.
We now want to use the usual Schwarz Lemma argument to show that on $X\times[0,T)$ we have
\begin{equation}\label{sc}
\omega(t)\geq C^{-1}f^*\omega_Y.
\end{equation}
To prove this, we first claim that at every point where $\tr{\omega(t)}{(f^*\omega_Y)}>0$ we have
\begin{equation}\label{sc2}
\left(\frac{\de}{\de t}-\Delta\right)\log\tr{\omega(t)}{(f^*\omega_Y)}\leq C\tr{\omega(t)}{(f^*\omega_Y)}.
\end{equation}
To see this, recall that by definition of a K\"ahler metric on a compact K\"ahler space, given every point $y\in Y$ there exists an open neighborhood $U$ of $y$ in $Y$ with an embedding
$U\hookrightarrow\mathbb{C}^N$, and a smooth and strictly plurisubharmonic function $\eta$ on $\mathbb{C}^N$, such that $\omega_Y$ equals the restriction of $\ddbar\eta$ to $U$. Clearly if $y$ is a smooth point of $Y$ this just says that $\omega_Y$ is a usual K\"ahler metric near $y$. Therefore if $x\in X$ is a point with $\tr{\omega(t)}{(f^*\omega_Y)}(x)>0$ and $f(x)$ is a smooth point of $Y$, then \eqref{sc2} holds near $x$ thanks to a standard ``Schwarz Lemma'' calculation (see e.g. \cite[Theorem 3.2.6]{SW}).

If on the other hand we have $\tr{\omega(t)}{(f^*\omega_Y)}(x)>0$ but $f(x)$ is a singular point of $Y$, then we choose a neighborhood $U$ of $f(x)$ as above, so that
$\omega_Y$ equals the restriction of $\ddbar\eta$ to $U$. On $f^{-1}(U)$ we then have the composite holomorphic map $\ti{f}:f^{-1}(U)\to \mathbb{C}^N$ of $f$ and the local embedding, such that on $f^{-1}(U)$ we have $f^*\omega_Y=\ti{f}^*\ddbar\eta$. Then we can apply the same Schwarz Lemma calculation as in \cite[Theorem 3.2.6]{SW}, to the
holomorphic map between K\"ahler manifolds $\ti{f}:(f^{-1}(U),\omega(t))\to (\mathbb{C}^N,\ddbar\eta)$, and \eqref{sc2} then holds on $f^{-1}(U)$.

On the other hand we also have
$$\left(\frac{\de}{\de t}-\Delta\right)\vp(t)=\dot{\vp}(t)-n+\tr{\omega(t)}{\hat{\omega}_t}\geq\log\frac{\omega(t)^n}{\Omega}-n+\frac{1}{4}\tr{\omega(t)}{(f^*\omega_Y)}+\frac{1}{2}\tr{\omega(t)}{\hat{\omega}_t},$$
provided that $t$ is sufficiently close to $T$, which we may always assume.
Therefore, if we choose $A$ large enough, we have that
\[\begin{split}
\left(\frac{\de}{\de t}-\Delta\right)(\log\tr{\omega(t)}{(f^*\omega_Y)}&-A\vp(t))\\
&\leq -\tr{\omega(t)}{(f^*\omega_Y)}-\tr{\omega(t)}{\hat{\omega}_t}-A\log\frac{\omega(t)^n}{\Omega}+An,
\end{split}\]
at all points where $\tr{\omega(t)}{(f^*\omega_Y)}>0$. Therefore
\[\begin{split}
\left(\frac{\de}{\de t}-\Delta\right)&\bigg(\log\tr{\omega(t)}{(f^*\omega_Y)}-A\vp(t)-An(T-t)(\log(T-t)-1)\bigg)\\
&\leq -\tr{\omega(t)}{(f^*\omega_Y)}-\tr{\omega(t)}{\hat{\omega}_t}-A\log\frac{\omega(t)^n}{(T-t)^n\Omega}+An,
\end{split}\]
which combined with
$$\tr{\omega(t)}{\hat{\omega}_t}\geq\frac{T-t}{T}\tr{\omega(t)}{\omega_0}\geq C^{-1}\left(\frac{(T-t)^n\Omega}{\omega(t)^n}\right)^{\frac{1}{n}}
\geq A\log\frac{(T-t)^n\Omega}{\omega(t)^n}-C,$$
gives
\[\begin{split}
\left(\frac{\de}{\de t}-\Delta\right)&\bigg(\log\tr{\omega(t)}{(f^*\omega_Y)}-A\vp(t)-An(T-t)(\log(T-t)-1)\bigg)\\
&\leq -\tr{\omega(t)}{(f^*\omega_Y)}+C,
\end{split}\]
and a simple application of the maximum principle then gives \eqref{sc}.

Lastly, using \eqref{sc} we can use the same argument as in \cite[Theorem 4.1]{So} and conclude that if $B$ is a geodesic ball of $\omega_Y$ contained in the regular part of $Y$ then the diameter of $f^{-1}(B)$ (as a subset of $X$) with respect
to $\omega(t)$ is bounded below by a small multiple of the diameter of $B$ with respect to $\omega_Y$, and hence remains bounded uniformly away from zero as $t\to T$,
a contradiction to our assumption \eqref{point}.
\end{proof}

To close, we make two simple observations.
\begin{remark}Assuming that \eqref{point} holds, then the main result of \cite{IK} shows that $H^1(X,\mathbb{R})=0$. This of course is consistent with Conjecture \ref{main}, since Fano manifolds have vanishing first Betti number.
\end{remark}
\begin{remark}
To prove Conjecture \ref{main} in general, it would be enough to show that if \eqref{point} holds then the $L^1$-type norm
$$\int_X |\omega(t)|_{\omega_0}\omega_0^n,$$
or the equivalent quantity
$$\int_X\omega(t)\wedge\omega_0^{n-1},$$
goes to zero as $t\to T$.
Indeed, any of these would imply that
$$\int_X\alpha\wedge\omega_0^{n-1}=0,$$
and the Khovanskii-Teissier inequality for nef classes (see e.g. \cite{FX})
$$\int_X\alpha\wedge\omega_0^{n-1}\geq\left(\int_X\alpha^2\wedge\omega_0^{n-2}\right)^{\frac{1}{2}}\left(\int_X\omega_0^n\right)^{\frac{1}{2}},$$
implies that $\int_X\alpha^2\wedge\omega_0^{n-2}=0.$ The result now follows from the Hodge-Riemann bilinear relations on K\"ahler manifolds, proved in \cite{DN}. Indeed, following their notation, we set
$\omega_1=\dots=\omega_{n-1}:=\omega_0$, so that the condition $\int_X\alpha\wedge\omega_0^{n-1}=0$ says that $\alpha\in P^{1,1}(X)$, while the condition
$\int_X\alpha^2\wedge\omega_0^{n-2}=0$ says that $Q(\alpha,\alpha)=0$. Since by \cite[Theorem A]{DN} the bilinear form $Q$ is positive definite on $P^{1,1}(X)$, this
implies that $\alpha=0$, as required.
\end{remark}

\end{document}